\documentclass[small]{article}

\usepackage{amsmath,amstext,amsthm,amsfonts}
\usepackage{makeidx}

\usepackage{marginnote}
\usepackage{bbding}

\usepackage{wasysym}
\usepackage{amssymb,latexsym}
\usepackage{txfonts, pxfonts}
\usepackage[english]{babel}
\usepackage[pdftex]{graphicx}
\usepackage{xcolor}
\numberwithin{equation}{section}
\newtheorem{theorem}{\textbf{Theorem}}[section]
\newtheorem{proposition}[theorem]{\textbf{Proposition}}

\newtheorem{lemma}[theorem]{\textbf{Lemma}}
\newtheorem{corollary}[theorem]{Corollary}

\theoremstyle{definition}
\newtheorem{definition}[theorem]{\textbf{Definition}}
\theoremstyle{remark}
\newtheorem{remark}[theorem]{\it{Remark}}
\newenvironment{notation}[1][Notation.]{\begin{trivlist}
\item[\hskip \labelsep {\bfseries #1}]}{\end{trivlist}}

\def\e{\epsilon}
\def\ei{\epsilon_i}
\def\R{\mathbb{R}}
\def\Rn{{\mathbb{R}}^n_+}

\def\d{\partial}

\def\fermi{\psi_i:B^+_{\delta}(0)\to M}
\def\fermilinha{\psi_i:B^+_{\delta'}(0)\to M}
\def\a{\alpha}
\def\b{\beta}

\def\l{\lambda}

\def\Sei{S^+_{\delta\ei^{-1}}}
\def\Bei{B^+_{\delta\ei^{-1}}}
\def\Dei{D_{\delta\ei^{-1}}}
\def\Beilinha{B^+_{\delta'\ei^{-1}}}
\def\Deilinha{D_{\delta'\ei^{-1}}}

\def\ba{\begin{align}}
\def\ea{\end{align}}
\def\bp{\begin{proof}}
\def\ep{\end{proof}}

\def\cmedia{h}
\def\ds{d\sigma}

\renewcommand{\(}{\left(}
\renewcommand{\)}{\right)}

\begin{document}

\title{A compactness theorem for conformal metrics with constant scalar curvature and constant boundary mean curvature in dimension three}
\date{}

\author{\textsc{S\'ergio Almaraz}\footnote{Partially supported by grants 202.802/2019 and 201.049/2022, FAPERJ/Brazil.}  \textsc{and Shaodong Wang}\footnote{Partially supported by NSFC 12001364.}}

\maketitle

\begin{abstract}

On a compact three-dimensional Riemannian manifold with boundary, we prove the compactness of the full set of conformal metrics with positive constant scalar curvature and constant mean curvature on the boundary. This involves a blow-up analysis of a Yamabe equation with critical Sobolev exponents both in the interior and on the boundary.

\end{abstract}




\section{Introduction}\label{sec:intr}

Let $(M,g)$ be a n-dimensional Riemannian manifold with boundary $\d M$, $n\geq 3$, and let $\nabla$ be its Riemannian connection. Denote by $R_g$ its scalar curvature and by $\Delta_g$ its Laplace-Beltrami operator, which is the Hessian trace. By $h_g$ we denote the boundary mean curvature with respect to the inward normal vector $\eta$, i.e. $h_g=-\frac{1}{n-1}\rm{div}_g\eta$.
In this paper we study the question of compactness of the full set of positive solutions to the equations
\begin{align}\label{main:equation}
\begin{cases}
L_{g}u+Ku^p=0,&\text{in}\:M,
\\
B_{g}u+cu^{\frac{p+1}{2}}=0,&\text{on}\:\partial M,
\end{cases}
\end{align}
where $1< p\leq (n+2)/(n-2)$ and $K, c\in \mathbb R$. Here, $L_g=\Delta_g-\frac{n-2}{4(n-1)}R_g$ is the conformal Laplacian and $B_g=\frac{\d}{\d\eta}-\frac{n-2}{2}h_g$ is the conformal boundary operator.

These equations have a very interesting geometrical meaning when $p=(n+2)/(n-2)$. A smooth solution $u>0$ of \eqref{main:equation} represents a conformal metric $\tilde g=u^{\frac{4}{n-2}}g$ with scalar curvature $R_{\tilde g}=4(n-1)K/(n-2)$ and boundary mean curvature $h_{\tilde g}=2c/(n-2)$, as \eqref{main:equation} becomes a particular case of the well known equations
\begin{align*}
\begin{cases}
L_{g}u+\frac{n-2}{4(n-1)}R_{\tilde g}u^{\frac{n+2}{n-2}}=0,&\text{in}\:M,
\\
B_{g}u+\frac{n-2}{2}h_{\tilde g}u^{\frac{n}{n-2}}=0,&\text{on}\:\partial M.
\end{cases}
\end{align*}

For $c$ non-negative,  the equations \eqref{main:equation} have a variational formulation in terms of the functional
\begin{align*}
Q_{c,p}(u)=\frac{
\int_M\(|\nabla_gu|^2+\frac{n-2}{4(n-1)}R_gu^2\)dv_g+\frac{n-2}{2}\int_{\d M}h_gu^2d\sigma_g
}{
\left(\int_{M}|u|^{p+1}dv_g\right)^{\frac{2}{p+1}}
+c\left(\int_{\d M}|u|^{\frac{p+3}{2}}d\sigma_g\right)^{\frac{4}{p+3}}
}\,,
\end{align*}
where $dv_g$ and  $d\sigma_g$ denote the volume and area forms of $M$ and $\d M$ respectively. A function $u$ is a critical point for $Q_{c,p}$ if and only if it solves \eqref{main:equation}. However, direct methods fail to work when $p=(n+2)/(n-2)$ as $p+1=2n/(n-2)$ and $(p+3)/2=2(n-1)/(n-2)$ are critical for the Sobolev embeddings $H^1(M)\hookrightarrow L^{p+1}(M)$ and $H^1(M)\hookrightarrow L^{\frac{p+3}{2}}(\d M)$ respectively.
Other variational formulations have been used in \cite{chen-ruan-sun, escobar4, han-li1, han-li2}.

Following \cite{chen-sun} we define the conformal invariant 
\begin{align*}
Q_c(M)=\inf\,\big\{Q_{c, \frac{n+2}{n-2}}(u):\:u\in C^1({M}), u\nequiv 0 \:\text{in}\: M\big\}\,,
\end{align*}
which is always finite and whose sign is independent of the constant $c\geq 0$. The constant $Q_c(M)$ has the same sign of the first eigenvalue $\l_1(L_g)$ of the problem
\begin{align*}
\begin{cases}
L_{g}u+\l u=0,&\text{in}\:M,
\\
B_{g}u=0,&\text{on}\:\partial M\,.
\end{cases}
\end{align*}
We say that $M$ is of positive, negative, or zero type if $\l_1(L_g)>0$, $\l_1(L_g)<0$, or $\l_1(L_g)=0$ respectively. 

The existence of conformal metrics with constant scalar curvature and constant boundary mean curvature was first studied by Escobar in \cite{escobar3, escobar2,  escobar4} (see also \cite{araujo}) motivated by the classical Yamabe problem on closed manifolds, while the regularity of solutions was obtained by Cherrier in \cite{cherrier}. 
The case when $K=0$, which corresponds to scalar flat metrics, is also known as the Escobar-Riemann problem (See \cite{ahmedou, almaraz1, chen, escobar3, coda1, coda2}). The case $c=0$, which corresponds to minimal boundary, is also of interest (see \cite{brendle-chen, escobar2}).
In either case, the module of the non-zero constant can be prescribed by simply multiplying the metric by a positive constant.  On the other hand, the existence of solutions when both $K$ and $c$ are non-zero is more subtle as a different approach is needed in order to prescribe both constants.
When $M$ is of positive type, the search for solutions with $K>0$ with any given $c\in\mathbb R$ is known as the Han-Li conjecture (see \cite{chen-ruan-sun, han-li1, han-li2}). 

In this paper we are interested in the setting of the Han-Li conjecture. When $M$ is a three-dimensional Riemannian manifold of positive type we study the compactness of the full set of solutions to the equations \eqref{main:equation} when $K>0$ and $c\in\mathbb R$ (see Remark \ref{rmk:neg} below for a discussion on the case $K<0$).
We will adopt the normalization $K=n(n-2)$. In order to have a conformally invariant version of \eqref{main:equation} we will work with the equations
\begin{align}\label{main:eq:conf}
\begin{cases}
L_{g}u+n(n-2)f^{-\tau} u^p=0,&\text{in}\:M,
\\
B_{g}u+c\bar{f}^{-\frac{\tau}{2}}u^{\frac{p+1}{2}}=0,&\text{on}\:\partial M,
\end{cases}
\end{align}
where $f$ and $\bar{f}$ are smooth positive functions on $M$ and $\partial M$ respectively, and $\tau=(n+2)/(n-2)-p$.
For most of the proofs we will be focusing on the critical case $p=(n+2)/(n-2)$, which is particularly challenging, so that the equations \eqref{main:eq:conf}  become
\begin{align}\label{main:eq:crit}
\begin{cases}
L_{g}u+n(n-2)u^{\frac{n+2}{n-2}}=0,&\text{in}\:M,
\\
B_{g}u+cu^{\frac{n}{n-2}}=0,&\text{on}\:\partial M.
\end{cases}
\end{align}

The compactness of the full set of solutions to \eqref{main:eq:crit} was proved for locally conformally flat manifolds with umbilical boundary in \cite{han-li1}. We refer the reader to \cite{almaraz-queiroz-wang, almaraz3,  ahmedou-felli1, ahmedou-felli2, ghimenti-micheletti1, ghimenti-micheletti2, kim-musso-wei}
for compactness results in the case $K=0$ and to \cite{disconzi-khuri} for the case $c=0$.

Our main result establishes compactness in full generality in dimension three for the geometric equation \eqref{main:eq:crit} and is analogous to the result in \cite{almaraz-queiroz-wang} where the case $K=0$ was studied. 
\begin{theorem}\label{compactness:thm}
	Let $(M,g)$ be a Riemannian three-manifold with boundary $\d M$.  Suppose that $M$ is of positive type and it is not conformally equivalent to the round hemisphere. 
	Then there exists $C(M,g,c)>0$ such that for any solution $u>0$ of (\ref{main:eq:crit})
	we have
	$$C^{-1}\leq u\leq C\:\:\:\:\: \text{and}\:\:\:\:\:\|u\|_{C^{2,\a}(M)}\leq C\,,$$
	for some $0<\a<1$.
\end{theorem}

We observe that the round hemisphere is conformally equivalent to any spherical cap, which has a family of blowing-up solutions (see Section \ref{sec:pre}). As a result, the hypothesis of $M$ not being conformally equivalent to the hemisphere is not restrictive.

Once we have proved Theorem \ref{compactness:thm}, the extension of that to more general exponents $1<p\leq \frac{n+2}{n-2}$ follows from standard methods (see \cite{almaraz-queiroz-wang, han-li1, li-zhu2, marques}). Hence, we state the following theorem whose proof will be omitted:

\begin{theorem}\label{compactness:thm:general}
Let $(M,g)$ be a Riemannian three-manifold with boundary $\d M$.  Suppose that $M$ is of positive type and is not conformally equivalent to the round hemisphere. 
Then, given $\bar c>0$ and a small $\gamma_0>0$, there exists $C(M,g,\bar c, \gamma_0)>0$ such that for any  $|c|\leq \bar c$, $p\in\left[1+\gamma_0,\frac{n+2}{n-2}\right]$, and any solution $u>0$ of (\ref{main:eq:conf})
we have
$$C^{-1}\leq u\leq C\:\:\:\:\: \text{and}\:\:\:\:\:\|u\|_{C^{2,\a}(M)}\leq C\,,$$
for some $0<\a<1$.
\end{theorem}

It is well known that a priori estimates play important roles in establishing the existence of solutions. As a corollary of Theorem \ref{compactness:thm:general}, one can obtain existence of solution to equations \eqref{main:equation} through topological degree theory (see for example the introduction section of \cite{han-li1}). This provides an an alternate proof in dimension three for the Han-Li conjecture which was proved in this dimension in \cite{chen-ruan-sun} using the mountain pass lemma.

Theorems \ref{compactness:thm} and \ref{compactness:thm:general}, as well as several other results on compactness of the full set of solutions to the Yamabe equations on manifolds with boundary, are highly inspired by the case of closed manifolds. In this case, the compactness for the classical Yamabe equation was originally conjectured by Schoen who proposed an approach based on a Pohozaev identity which relates interior integral terms with boundary integral terms in a neighbourhood of a blow-up point. This method was successfully implemented by Khuri-Marques-Schoen \cite{khuri-marques-schoen} to prove the compactness conjecture of Yamabe problems on manifolds without boundary in dimension $3\leq n\leq 24$ while counterexamples in dimension $n\geq 25$ were provided by Brendle \cite{brendle2} and Brendle-Marques \cite{brendle-marques} (we refer the reader to \cite{berti-malchiodi} for non-smooth counterexamples).

A crucial step in the argument in \cite{khuri-marques-schoen} is, after assuming by contradiction the existence of a blow-up sequence, to express this sequence as a canonical Euclidean bubble plus a correction term near a blow-up point. This correction term is defined as a solution to a linear problem involving coefficients of the metric expansion at the blow-up point and allows the authors to identify a quadratic form in those coefficients. The correction term not only makes the quadratic form positive definite but also gives a sharp approximation for the blow-up sequence. 
On the other hand, the Positive Mass Theorem applied to the boundary integral terms of the Pohozaev identity gives a sign that contradicts the positivity of the quadratic form which comes from interior integral terms. 

While in the case of closed manifolds the correction term is unnecessary in low dimensions (see \cite{druet2, li-zhang, li-zhang2, li-zhu2, marques}), this situation is different in the case of manifolds with boundary (see \cite{almaraz-queiroz-wang, kim-musso-wei}). To be more concrete, Proposition 5.3 in \cite{almaraz-queiroz-wang} only provides order one approximation near a blow-up point. Although this is enough in the dimension three, which is the dimension addressed in \cite{almaraz-queiroz-wang}, in the case of closed manifolds such a low order approximation can be easily obtained without relying on the correction term used in \cite{almaraz-queiroz-wang} (see \cite{li-zhu2, marques}). 

This is one of the reasons why the compactness problem on manifolds with boundary is technically more involved than its counterpart on closed manifolds. Another reason is that the linear problem determining the correction term cannot be solved in the class of polynomials in the coefficients of the metric. However, we show that in dimension three such an explicit solution is not necessary as we are able to include the part coming from the correction term in a high order term. In that dimension this high order term also includes the interior integral term coming from the Pohozaev identity.
 This is where the dimensional hypothesis is used in our argument.

The case $K>0$ in the equations \eqref{main:equation} imposes some extra difficulties when compared to the case $K=0$, studied in \cite{almaraz-queiroz-wang}. While in \cite{almaraz-queiroz-wang} the maximum points for blow-up sequences are always sitting on the boundary, in the case $K>0$ we have to handle also interior local maxima. Although interior blow-up points accumulating in the interior of the manifold are handled as in the case of closed manifolds, and the case of boundary maximum points is similar to the one in \cite{almaraz-queiroz-wang}, the case of interior local maximum points converging to the boundary has to be more carefully addressed. The reason is that the canonical bubbles that best approximate the solution near to these points may have slightly different centres than the local maximum points for the equations solutions. This problem is overcame by some modification in the linear correction term (see equations \eqref{hip:phi} and compare with \cite[Equations (5.4)]{almaraz-queiroz-wang}).

Similar to the case of closed manifolds, in the case of the equations \eqref{main:eq:crit} we also expect compactness to hold up to certain dimension and to fail in greater dimensions. Chen-Wu \cite{chen-wu} constructed blow-up examples for the equations \eqref{main:eq:crit} in dimension $n\geq 62$ (in the case $c=0$ such examples were provided in dimension $n\geq 25$ by Disconzi-Khuri in \cite{disconzi-khuri}). While compactness for these equations was proved in all dimensions for locally conformally flat manifolds with umbilical boundary by Han-Li \cite{han-li1}, the problem for general manifolds remains open and, to the best of our knowledge, Theorems \ref{compactness:thm} and \ref{compactness:thm:general} are the first general results for this setting.  

\begin{remark}\label{rmk:neg}
Let us briefly discuss the compactness when $K<0$ in equations \eqref{main:equation}. In the case $c<n-2$ it is proven in \cite{han-li1} that compactness holds in any dimension $n\geq 3$ with no additional hypothesis. The argument relies on the fact that there is no solution for the corresponding Euclidean equation on the half-space. This rules out the blow-up phenomenon because a standard rescaling argument for blow-up sequences leads to this Euclidean equation in the limit. However, the same argument would fail for the case $c\geq n-2$ where hyperbolic geodesic balls and horopherical domains provide limit solutions in the Euclidean half-space (see \cite{chipot-fila-shafrir} and also \cite{escobar1, li-zhu, terracini}). We will investigate this case in an upcoming work.
\end{remark}

This paper is organized as follows. In Section \ref{sec:pre} we present some preliminary results about the standard solution on the Euclidean half-space, Fermi coordinates, the Pohozaev identity and the Positive Mass Theorem for manifolds asymptotically modelled on the Euclidean half-space. The definition of isolated and isolated simple blow-up points and some additional properties are collected in Section \ref{sec:isolated:simple}, while the refined blow-up estimates are presented in Section \ref{sec:blowup:estim}. In Section \ref{sec:sign:restr} we prove a Pohozaev sign restriction using the estimates obtained in Section \ref{sec:blowup:estim} and relate it to the mass term using a flux integral. Finally we give a proof of Theorem \ref{compactness:thm} in Section \ref{sec:pf:thm}.


\section{Preliminaries}\label{sec:pre}

\begin{notation}
Throughout this work we will make use of the index notation for tensors, commas denoting covariant differentiation. We will adopt the summation convention whenever confusion is not possible. When dealing with coordinates on manifolds with boundary, we will use indices $1\leq i,j,k,l\leq n-1$ and $1\leq a,b,c,d\leq n$. 

We will denote by $g$ the Riemannian metric and the induced metric on $\d M$ will be denoted by $\bar{g}$. We will denote by by $\Delta_g$ the Laplacian-Beltrami operator and by $R_g$ we will denote the scalar curvature. The second fundamental form of the boundary will be denoted by $\pi_{kl}$ and the mean curvature,  $\frac{1}{n-1}tr (\pi_{kl})$, by $h_g$.

By $\Rn$ we will denote the half-space $\{z=(z_1,...,z_n)\in \R^n;\:z_n\geq 0\}$. If $z\in\Rn$ we set $\bar{z}=(z_1,...,z_{n-1})\in\R^{n-1}\cong \d\Rn$. We define 

$$B^+_{\delta}(0)=\{z\in\Rn\,;\:|z|<\delta \}, \qquad  S^+_{\delta}(0)=\{z\in\Rn\,;\:|z|=\delta \},$$
and $$D_{\delta}(0)=B^+_{\delta}(0)\cap \d\Rn=\{z\in\d\Rn\,;\:|z|<\delta \}.$$
Thus, $\d B^+_{\delta}(0)=D_{\delta}(0)\cup S^+_{\delta}(0)$.
We also denote $B^+_{\delta}=B^+_{\delta}(0)$,  $S^+_{\delta}=S^+_{\delta}(0)$  and $D_{\delta}=D_{\delta}(0)$ for short. 
Similarly, we define on the manifold $M$ the geodesic balls $B_\delta(x_i)$ and the geodesic spheres $S_\delta(x_i)=\partial B_\delta(x_i)\cap (M\backslash \partial M)$ for any $x_i\in M$.
\end{notation}

We begin this section with the simplest example of solution to the Yamabe-type problem we are concerned, which are the spherical caps in $\R^{n+1}$ endowed with the canonical metric. Let $F$ be the stereographic projection from the unit sphere in $\R^{n+1}$ centred at $(0,..., 0, T, 0)$, $T\in \R$, onto the hyperplane $\R^n=\{\xi_{n+1}=0\}$ and let $(\xi_1,...,\xi_{n+1})$ be the coordinates taking $(0,..,0, T, 0)$ as its origin and $(y_1,...,y_n)$ be the coordinates of $\R^n$. 
Then

$$\xi_j=\frac{2y_j}{1+|\bar y|^2+(y_n-T)^2}\,,
\qquad j=1,...,n-1\,,
$$
$$
\xi_n=\frac{2(y_n-T)}{1+|\bar y|^2+(y_n-T)^2}\,,
\qquad
\xi_{n+1}=\frac{|\bar y|^2-1+(y_n-T)^2}{1+|\bar y|^2+(y_n-T)^2}\,.
$$
Let $\Sigma$ be the cap $F^{-1}(\R_+^n)$ on $\mathbb S^n$ and $\Gamma_T(\xi_1,...,\xi_{n+1})=(\tilde\xi_1,...,\tilde\xi_{n+1})$ be the rotation given by 
$$
\tilde\xi_j=\xi_j\,,\qquad j=1,...,n-1,
$$
$$
\tilde\xi_n=\frac{1}{\sqrt{1+T^2}}(-T\xi_n-\xi_{n+1}),
$$
$$
\tilde\xi_{n+1}=\frac{1}{\sqrt{1+T^2}}(\xi_n-T\xi_{n+1}).
$$
A direct calculation shows that $\mathbb{S}^n_T:=\Gamma_T(\Sigma)$ is given by
$$
\mathbb S_T^n=\big\{(\tilde\xi_1,...,\tilde\xi_{n+1})\in\R^{n+1}\:|\:\tilde\xi_{n+1}\geq \frac{-T}{\sqrt{1+T^2}}\big\}.
$$
Then $\mathbb{S}^n_T$ is conformally equivalent to the Euclidean half space $\R_+^n$ through the map
$$
F_T:=F\circ\Gamma_T^{-1}:\mathbb S_T^n\to \mathbb R^n_+.
$$

A direct calculation shows that the induced metric on $\R_+^n$ is $4U_T^{\frac{4}{n-2}}g_{eucl}$ where $g_{eucl}$ is the Euclidean metric and   
$$
U_T(y)=\(\frac{1}{1+\sum_{j=1}^{n-1}y_j^2+(y_n-T)^2}\)^{\frac{n-2}{2}}.
$$ 
Here, the constant $-T$ is the mean curvature on the boundary of the cap. 
The function $U_T$ satisfies
\begin{equation}
\label{eq:U}
\begin{cases}
\Delta U_T + n(n-2)U_T^{\frac{n+2}{n-2}} = 0\,,&\text{in}\:\mathbb{R}_+^n\,,\\
\displaystyle\frac{\partial U_T}{\partial y_n}+cU_T^{\frac{n}{n-2}}=0\,,&\text{on}\:\partial\mathbb{R}_+^n\,,
\end{cases}
\end{equation}
where $c=-(n-2)T$.
Since the equations (\ref{eq:U}) are invariant by horizontal translations and scaling with respect to the origin, we obtain the following family of solutions of (\ref{eq:U}):
\begin{equation}\label{fam:U}
\lambda^{\frac{2-n}{2}}U_{T}(\lambda^{-1}(y_1-z_1, ..., y_{n-1}-z_{n-1}, y_n))=\left(\frac{\l}{\lambda^2+\sum_{j=1}^{n-1}(y_j-z_j)^2+(y_n-\l T)^2}\right)^{\frac{n-2}{2}}\,,
\end{equation}
where $\l>0$ and $(z_1,...,z_{n-1})\in \R^{n-1}$.

In fact, the converse statement is also true: by a Liouville-type theorem in \cite{li-zhu} (see also \cite{chipot-fila-shafrir, escobar1}), any non-negative solution to the equations (\ref{eq:U}) is either of the form (\ref{fam:U}) or is identically zero.


The existence of the family of solutions (\ref{fam:U}) has two important consequences. Firstly, we see that the set of solutions to the equations (\ref{eq:U}) is non-compact. In particular, the set of solutions of (\ref{main:eq:crit}) with $p=(n+2)/(n-2)$ is non-compact when $M$ is conformally equivalent to the round hemisphere. Secondly, the functions $J_j:=\d U_T/\d y_j$, for $j=1,...,n-1$, and $J_n:=\frac{n-2}{2}U_T+\sum_{b=1}^{n}y^b\d U_T/\d y^b$, are solutions to the following homogeneous linear problem:
\ba
\begin{cases}\label{linear:homog}
\Delta\psi+n(n+2)U_T^{\frac{4}{n-2}}\psi=0\,,&\text{in}\:\Rn\,,
\\
\displaystyle\frac{\d\psi}{\d y_n}-nTU_T^{\frac{2}{n-2}}\psi=0\,,&\text{on}\:\d\Rn\,.
\end{cases}
\end{align}

As a matter of fact, a converse statement is also true as stated in the next lemma that will be used later in Section \ref{sec:blowup:estim}.

\begin{lemma}
\label{classifLinear}
Suppose $\psi$ is a solution of \eqref{linear:homog}.
If $\psi(y)=O((1+|y|)^{-\alpha})$ for some $\alpha>0$, then there exist constants $c_1,...,c_{n}$ such that
$$
\psi (y)=\sum_{j=1}^{n-1} c_j\frac{\d U_T}{\d y_j}+c_n\Big(\frac{n-2}{2}U_T+\sum_{b=1}^ny^b\frac{\d U_T}{\d y^b}\Big)\,.
$$
\end{lemma}
\bp
The proof is similar to \cite[Lemma 2.1]{almaraz3}. For the reader's convenience, we include some details here. Recall that the operators $L_g$ and $B_g$ satisfy the following conformal laws
$$
L_{\xi^{\frac{4}{n-2}}g}(\xi^{-1}u)=\xi^{-\frac{n+2}{n-2}}L_gu
\qquad\text{and}\qquad
B_{\xi^{\frac{4}{n-2}}g}(\xi^{-1}u)=\xi^{-\frac{n}{n-2}}B_gu,
$$
for any smooth functions $\xi>0$ and $u$.
Hence, the equations \eqref{linear:homog} are equivalent to 
\begin{align*}
\begin{cases}
\Delta_{\mathbb S^n} \bar{\psi} + n\bar \psi= 0\,,&\text{in}\:\:\mathbb S_T^n\backslash\{(0,...,0,1)\},\\
\displaystyle\frac{\partial \bar{\psi}}{\partial \eta}-T\bar{\psi}=0\,,&\text{on}\:\:\partial\mathbb  S_T^n\backslash\{(0,...,0,1)\},
\end{cases}
\end{align*}
with $\bar\psi=(2^{\frac{2-n}{2}}U_T^{-1}\psi)\circ F_T$. The hypothesis $\psi(y)=O((1+|y|)^{-\alpha})$ implies that the singularity at $(0,...,0,1)$ is removable (see \cite[Lemma 2.7]{almaraz5}) so that $\bar\psi$ is a solution to  
\begin{align*}
\begin{cases}
\Delta_{\mathbb S^n} \bar{\psi} + n\bar \psi= 0\,,&\text{in}\:\:\mathbb S_T^n,\\
\displaystyle\frac{\partial \bar{\psi}}{\partial \eta}-T\bar{\psi}=0\,,&\text{on}\:\:\partial\mathbb  S_T^n.
\end{cases}
\end{align*}
The result then follows from Lemma \ref{eigenvalues} below.
\ep
\begin{lemma}\label{eigenvalues}
The eigenfunctions $\bar \psi$ corresponding to the eigenvalue $\mu=-T$ of the problem
\begin{equation}\label{linear:3'}
\begin{cases}
\Delta_{\mathbb S^n} \bar{\psi} + n\bar \psi= 0\,,&\text{in}\:\:\mathbb S_T^n,\\
\displaystyle\frac{\partial \bar{\psi}}{\partial \eta}+\mu\bar{\psi}=0\,,&\text{on}\:\:\partial\mathbb  S_T^n,
\end{cases}
\end{equation}
are the coordinate functions $\tilde\xi_{1}, ..., \tilde\xi_{n}$ restricted to $\mathbb S_T^n$. 
Here, $\Delta_{\mathbb S^n}$ denotes the Laplacian in $\mathbb S^n$ and $\eta$ is the inward unit normal vector to $\partial\mathbb S_T^n$.   
Moreover, $\tilde\xi_{j}\circ F_T^{-1}=\frac{-2}{n-2}U_T^{-1}J_j$, for $j=1,...,n-1$, and $\tilde\xi_{n}\circ F_T^{-1}=\frac{2}{(n-2)}\frac{1}{\sqrt{1+T^2}}U_T^{-1}J_n$.
\end{lemma}
\bp 
The first statement follows from \cite[Proposition 3.2]{han-li2} and the relation between the $\tilde\xi_{a}$ and the $J_a$ is obtained by direct computation.
\ep



Next we recall the definition of Fermi coordinates and some expansions for the metric. 

\begin{definition}\label{def:fermi}
Let $x_0\in\d M$ and choose boundary geodesic normal coordinates $(z_1,...,z_{n-1})$, centred at $x_0$, of the point $x\in\d M$.
We say that $z=(z_1,...,z_n)$, for small $z_n\geq 0$, are the {\it{Fermi coordinates}} (centred at $x_0$) of the point $\exp_{x}(z_n\eta(x))\in M$. Here, we denote by $\eta(x)$ the inward unit normal vector to $\d M$ at $x$. In this case, we have a map  $\psi(z)=\exp_{x}(z_n\eta(x))$, defined on a subset of $\Rn$.
\end{definition}
It is easy to see that in these coordinates $g_{nn}\equiv 1$ and $g_{jn}\equiv 0$, for $j=1,...,n-1$.

The expansion for $g$ in Fermi coordinates around $x_0$ is given by:
	\begin{align}\label{exp:g}
	g_{ij}(\psi(z))&=\delta_{ij}-2\pi_{ij}(x_0)z_n+O(|z|^2),\notag
	\\
	g^{ij}(\psi(z))&=\delta_{ij}+2\pi_{ij}(x_0)z_n+O(|z|^2),\notag
\\
{\rm{det}}\,g_{ab}\,(\psi(z))&=1-(n-1)h_g(x_0) z_n+O(|z|^2).
	\end{align}


Next we obtain a Pohozaev identity in our situation. Let $g$ be a Riemannian metric on the half-ball $B_{\delta}^+$ with $\partial B_{\delta}^+=S_{\delta}^+\cup D_{\delta}$. For any $z=(z_1,...,z_n)\in \R^n$ we set $r=|z|=\sqrt{z_1^2+...+z_n^2}$. For any smooth function $u$ on $B^+_{\delta}$ and $0<\rho<\delta$ we define
\begin{align*}
	P(u,\rho)&=\int_{S_{\rho}^+}\left(\frac{n-2}{2}u\frac{\partial u}{\partial r}-\frac{r}{2}|du|^2+r\left|\frac{\partial u}{\partial r}\right|^2\right)d\sigma
	+\frac{\rho}{p+1}\int_{S_{\rho}^+}Kf^{-\tau}u^{p+1}d\sigma \\&+\frac{2\rho}{p+3}\int_{\partial D_{\rho}}c\bar{f}^{-\frac{\tau}{2}}u^{\frac{p+3}{2}}d\bar{\sigma}
\end{align*}
and
$$P'(u,\rho)=\int_{S_{\rho}^+}\left(\frac{n-2}{2}u\frac{\partial u}{\partial r}-\frac{r}{2}|du|^2+r\left|\frac{\partial u}{\partial r}\right|^2\right)d\sigma\,.$$
\begin{proposition}\label{Pohozaev}
	If $u$ is a solution of
	\begin{equation}\notag
		\begin{cases}
			L_g u+Kf^{-\tau}u^{p}=0, \,&\text{in}\:B^+_{\rho}\,,\\
			B_gu+c\bar{f}^{-\frac{\tau}{2}}u^{\frac{p+1}{2}}=0, \,&\text{on}\:D_{\rho} \,,
		\end{cases}
	\end{equation}
	where $K, c$ are constants, then
	\begin{align}
		P(u,\rho)&=-\int_{B_{\rho}^+}\left(z^a\partial_a u+\frac{n-2}{2}u\right)(L_g-\Delta)(u)dz
		-\int_{D_{\rho}}\left(z^k\partial_k u+\frac{n-2}{2}u\right)(B_g-\partial_n)(u)d\bar{z}\notag
		\\
		&-\frac{\tau}{p+1}\int_{B^+_{\rho}}K(z^a\partial_a f)f^{-\tau-1}u^{p+1}d{z}
		+\left(\frac{n}{p+1}-\frac{n-2}{2}\right)\int_{
			B^+_{\rho}}Kf^{-\tau}u^{p+1}dz\,,\notag
		\\&-\frac{\tau}{p+3}\int_{D_{\rho}}c(z^k\partial_k \bar{f})\bar{f}^{-\frac{\tau}{2}-1}u^{\frac{p+3}{2}}d\bar{z}
		+\left(\frac{2(n-1)}{p+3}-\frac{n-2}{2}\right)\int_{D_{\rho}}c\bar{f}^{-\frac{\tau}{2}}u^{\frac{p+3}{2}}d\bar{z}\,,\notag
	\end{align} 
\end{proposition}
\begin{proof}
	The proof is similar to Proposition 3.1 of \cite{almaraz3} using integration by parts. 
\end{proof}


In the following we introduce a geometric invariant term called the mass and present a boundary version of the Positive Mass Theorem to be used in later sections. See \cite{almaraz-barbosa-lima} for a reference. 

\begin{definition}\label{def:asym}
	Let $(\hat M, g)$ be a Riemannian manifold with a non-compact  boundary $\d \hat M$. 
	We say that $\hat M$ is {\it{asymptotically flat}} with order $q>0$, if there is a compact set $K\subset \hat M$ and a diffeomorphism $f:\hat M\backslash K\to \Rn\backslash \overline{B^+_1}$ such that, in the coordinate chart defined by $f$ (which we call the {\it  asymptotic coordinates} of $\hat M$), we have
	$$
	|g_{ab}(y)-\delta_{ab}|+|y||g_{ab,c}(y)|+|y|^2|g_{ab,cd}(y)|=O(|y|^{-q})\,,
	\:\:\:\:\text{as}\:\:|y|\to\infty\,,
	$$
	where $a,b,c,d=1,...,n$.
\end{definition}

Suppose the manifold $\hat M$, of dimension $n\geq 3$,  is asymptotically flat with order $q>\frac{n-2}{2}$, as defined  above. Assume also that $R_g$ is integrable on $\hat M$, and $\cmedia_g$  is integrable on $\d \hat M$. Let $(y_1,...,y_n)$ be the  asymptotic coordinates induced by the diffeomorphism $f$. 
Then the limit
\begin{align}\label{def:mass}
	m(g)=
	\lim_{R\to\infty}\left\{
	\sum_{a,b=1}^{n}\int_{y\in\Rn,\, |y|=R}(g_{ab,b}-g_{bb,a})\frac{y_a}{|y|}\,\ds
	+\sum_{i=1}^{n-1}\int_{y\in\d\Rn,\, |y|=R}g_{ni}\frac{y_i}{|y|}\,d\bar\sigma\right\}
\end{align}
exists, and we call it the {\it mass} of $(\hat M, g)$. As proved in \cite{almaraz-barbosa-lima}, $m(g)$ is a geometric invariant in the sense that it does not depend on the asymptotic coordinates.

The expression in \eqref{def:mass} is the analogue of the ADM mass for the manifolds of Definition \ref{def:asym}. A Positive Mass Theorem for $m(g)$, similar to the classical ones in \cite{schoen-yau, witten}, is stated as follows:
\begin{theorem}[\cite{almaraz-barbosa-lima}]\label{pmt}
	Assume $3\leq n\leq 7$.
	If $R_g$, $\cmedia_g\geq 0$, then we have $m(g)\geq 0$ and the equality holds if and only if $\hat M$ is isometric to $\R_+^n$.
\end{theorem}

Theorem~\ref{pmt} will be used in Section 6 in the proof of Theorem~\ref{compactness:thm} to rule out the existence of blowing-up sequences.


\section{Isolated and isolated simple blow-up points}\label{sec:isolated:simple}

In this section we collect the definitions and main results on isolated and isolated simple blow-up sequences. We refer the reader to \cite{han-li1} for proofs and more details.
\begin{definition}\label{def:blow-up}
	Let $(M,g)$ be a compact Riemannian manifold with boundary and let $\Omega\subset M$ be a domain.
Let $\{u_i\}$ be a sequence satisfying  
\begin{align}\label{eq:blow-up}
\begin{cases}
L_{g_i}u_i+n(n-2)f_i^{-\tau_i} u_i^{p_i}=0,&\text{in}\:\Omega,
\\
B_{g_i}u_i+c\bar f_i^{-\frac{\tau_i}{2}}u_i^{\frac{p_i+1}{2}}=0,&\text{on}\:\Omega\cap \partial M,
\end{cases}
\end{align}
where each $g_i$ is a Riemannian metric on $M$, $f_i$ and $\bar f_i$ are smooth positive functions on $\Omega$ and $\Omega\cap\partial M$ respectively, $1<p_i\leq (n+2)/(n-2)$, $\tau_i=(n+2)/(n-2)-p_i$ and $c\in \mathbb R$. 
We say that $x_0\in \Omega$ is a {\it{blow-up point}} for $\{u_i\}$ if there is a sequence $\{x_i\}\subset \Omega$ such that 
	
	(1) $x_i\to x_0$;
	
	(2) $u_i(x_i)\to\infty$; 
	
	(3) $x_i$ is a local maximum of $u_i$.\\
	The sequence $\{x_i\}$ is called a {\it{blow-up sequence}}. Briefly we say that $x_i\to x_0$ is a blow-up point for $\{u_i\}$, meaning that $\{x_i\}$ is a blow-up sequence converging to the point $x_0$. We will always assume that $g_i$ and $f_i$ converge in $C^2(\Omega)$ to some $g_0$ and $f_0$, respectively.
\end{definition}
\begin{notation}
	If $x_i\to x_0$ is a blow-up point we set $\ei=u_i(x_i)^{\frac{1-p_i}{2}}\rightarrow 0$. We also set $T=-c/(n-2)$.
\end{notation}
\begin{definition}\label{def:isolado}
	We say that a blow-up point $x_i\to x_0$ is an {\it{isolated}} blow-up point for $\{u_i\}$ if there exist $\delta,C>0$ such that 
	\begin{equation}\notag
	u_i(x)\leq Cd_{g_i}(x,x_i)^{-\frac{2}{p_i-1}}\,,\:\:\:\:\text{for all}\: x\in M\backslash \{x_i\}\,,\:d_{g_i}(x,x_i)< \delta\,.
	\end{equation}
\end{definition}
	If the domain $\Omega\subset \mathbb R^n$ contains the origin, for any coordinate system $\psi_i:\Omega\to M$ centred at $x_i$ the above definition is equivalent to
	\begin{equation}\label{des:isolado}
	u_i(\psi_i(z))\leq C|z|^{-\frac{2}{p_i-1}}\,,\:\:\:\:\text{for all}\: z\in \Omega\backslash\{0\}\,.
	\end{equation}
	This definition is invariant under re-normalization. This follows from the fact that if $v_i(y)=s^{\frac{2}{p_i-1}}u_i(\psi_i(sy))$, then
	$$
	u_i(\psi_i(z))\leq C|z|^{-\frac{2}{p_i-1}}\Longleftrightarrow v_i(y)\leq C|y|^{-\frac{2}{p_i-1}}\,,
	$$
	where $z=sy$.

	Classical Harnack inequalities give the following lemma:
\begin{lemma}\label{Harnack}
	Let $x_i\to x_0$ be an isolated blow-up point and $\delta$ as in Definition \ref{def:isolado}. 
	Then there exists $C>0$ such that for any $0<s<\delta/3$ we have
	\begin{equation}\notag
	\max_{s/2<d_g(x,x_i)<2s} u_i\leq C\min_{s/2<d_g(x,x_i)<2s} u_i\,.
	\end{equation}
\end{lemma}

As a consequence we state the next proposition which says that, in the case of an isolated blow-up point, the sequence $\{u_i\}$, when renormalized, converges to the standard Euclidean solutions.
\begin{proposition}\label{form:bolha}
	Let $x_i\to x_0$ be an isolated blow-up point. We set
	$$
	v_i(y)=\ei^{\frac{2}{p_i-1}}(u_i\circ\psi_i)(\ei y)\,,\:\:\:\: \text{for}\: \psi_i(y)\in B_{\delta \ei^{-1}}(x_i)\,,
	$$
the geodesic ball with radius $\delta\ei^{-1}$ centred at $x_i$. Then given $R_i\to\infty$ and $\b_i\to 0$, after choosing subsequences, we have  
	\\\\
	(1) Either $x_i\in M\backslash\d M$ and $$\big\|v_i-\Big(\frac{1}{1+|y|^2}\Big)^{\frac{n-2}{2}}\big\|_{C^2(\psi_i^{-1}(B_{2R_i}(x_i)))}<\b_i,$$ 
or $x_i\in\d M$ and 
$$
\big\|v_i-\Big(\frac{\l}{\lambda^2+\sum_{j=1}^{n-1}y_j^2+(y_n-\lambda T)^2}\Big)^{\frac{n-2}{2}}\,\big\|_{C^2(\psi_i^{-1}(B_{2R_i}(x_i)))}<\b_i,
$$
where $\lambda=1/(1+T^2)$;
	\\\\
	(2) $\lim_{i\to\infty}\frac{R_i}{|\log\ei|}= 0$, and  $\lim_{i\to\infty}p_i=\frac{n+2}{n-2}$.
\end{proposition}

The set of blow-up points is handled in the next proposition, which is \cite[Proposition 1.1]{han-li1}.
\begin{proposition}\label{conj:isolados}
	Given small $\b>0$ and large $R>0$ there exist constants $C_0, C_1>0$, depending only on $\b$, $R$ and $(M,g)$, such that if $u$ is solution of \eqref{main:eq:conf} and $\max_{M} u\geq C_0$, then $(n+2)/(n-2)-p<\b$ and
	there exist $x_1,...,x_N\in M$, $N=N(u)\geq 1$, local maxima of $u$, such that:
	\\\\
	(1) If $r_j=Ru(x_j)^{-\frac{p-1}{2}}$ for $j=1,...,N$,  then $\{B_{r_j}(x_j)\subset M\}_{j=1}^{N}$ is a disjoint collection, where $B_{r_j}(x_j)$ is the geodesic ball.
	\\\\
	(2) For each $j=1,...,N$, either  $x_j\in M\backslash\d M$ and 
$$\big\|u(x_j)^{-1}u\big(\psi_j(u(x_j)^{\frac{1-p}{2}}y)\big)-\Big(\frac{1}{1+|y|^2}\Big)^{\frac{n-2}{2}}\big\|_{C^2(\psi_j^{-1}(B_{2R}(x_j)))}<\b,$$
or $x_j\in\d M$ and
$$
\big\|u(x_j)^{-1}u\big(\psi_j(u(x_j)^{\frac{1-p}{2}}y)\big)-\Big(\frac{\l}{\lambda^2+\sum_{j=1}^{n-1}y_j^2+(y_n-\lambda T)^2}\Big)^{\frac{n-2}{2}}\big\|_{C^2(\psi_j^{-1}(B_{2R}(x_j))}<\b,
$$
where $\lambda=1/(1+T^2)$.
\\\\
	(3) We have
	$$
	u(x)\,d_{g}(x,\{x_1,...,x_N\})^{\frac{2}{p-1}}\leq C_1\,,\:\:\:\text{for all}\: x\in M\,,
	$$
	$$
	u(x_j)\,d_{g}(x_j,x_k)^{\frac{2}{p-1}}\geq C_0\,,\:\:\:\:\text{for any}\: j\neq k\,,\:j,k=1,...,N\,.
	$$ 
\end{proposition}

We now introduce the notion of an isolated simple blow-up point. If $x_i\to x_0$ is an isolated blow-up point for $\{u_i\}$,  for $0<r<\delta$, set 
$$\bar{u}_i(r)=\frac{2}{\sigma_{n-1}r^{n-1}}\int_{d_{g_i}(x,x_i)=r} u_i d\sigma_r
\:\:\:\:\text{and}\:\:\:\:w_i(r)=r^{\frac{2}{p_i-1}}\bar{u}_i(r)\,.$$ 
Note that the definition of $w_i$ is invariant under re-normalization. More precisely, using geodesic coordinates $\psi_i$ centred at $x_i$, if $v_i(y)=s^{\frac{2}{p_i-1}}u_i(\psi_i(sy))$, then
$r^{\frac{2}{p_i-1}}\bar{v}_i(r)=(sr)^{\frac{2}{p_i-1}}\bar{u}_i(sr)$.
\begin{definition}\label{def:simples}
	An isolated blow-up point $x_i\to x_0$ for $u$ is {\it{simple}} if there exists $\delta>0$ such that $w_i$ has exactly one critical point in the interval $(0,\delta)$.
\end{definition}
\begin{remark}\label{rk:def:equiv} Let $x_i\to x_0$ be an isolated blow-up point and $R_i\to\infty$. Using Proposition \ref{form:bolha} it is not difficult to see that, choosing a subsequence, $r\mapsto r^{\frac{2}{p_i-1}}\bar{u}_i(r)$ has exactly one critical point in the interval $(0,r_i)$, where $r_i=R_i\ei \to 0$. Moreover, its derivative is negative right after the critical point. Hence, if $x_i\to x_0$ is isolated simple then there exists $\delta>0$ such that $w_i'(r)<0$ for all $r\in [r_i,\delta)$. 
\end{remark}

In the next lemma we present some elementary results for isolated blow up points $x_i\rightarrow x_0$ when $x_0\in \d M$:
\begin{lemma}\label{lem:bddTi}
Let $x_i\to x_0\in \partial M$ be an isolated blow-up point for $\{u_i\}$ with $p_i\to\frac{n+2}{n-2}$. 
If $T<0$, then $\{x_i\}\subset \partial M$, up to a subsequence.
If $T\geq 0$, setting $T_i=d_{g_i}(x_i,\partial M)u_i(x_i)^{\frac{p_i-1}{2}}$ we have $\lim_{i\to\infty}T_i=T$.
\end{lemma}
\bp 
We first prove that the sequence $\{T_i\}$ is bounded.
Assume by contradiction that $T_i\to\infty$ and let $\tilde\psi_i$ be local coordinates centred at $x_i$. Then we define the rescaled function
$$
\varphi_i(z)=d_{g_i}(x_i,\partial M)^{\frac{2}{p_i-1}} u_i \big(\tilde\psi_i^{-1}(d_{g_i}(x_i,\partial M)\, z)\big)
$$
and follow the lines of \cite[Lemma 2.1]{han-li1} to reach a contradiction by making use of the Pohozaev identity Proposition \ref{Pohozaev}. 

Without loss of generality we can assume that either the $x_i$ are interior critical points of $u_i$ or $\{x_i\}\subset\partial M$ satisfies $\partial u_i/\partial \eta(x_i)\leq 0$. 

For large $i$, choose $\bar x_i$ to be the point on $\partial M$ closest to $x_i$ and let $\psi_i: B^+_{\delta'}\to M$ denote Fermi coordinates centred at $\bar{x}_i$. We set $\xi_i=\epsilon_i^{-1}\psi_i(x_i)\in\mathbb R^n_+$ and observe that $\xi_i=(0,...,0,T_i)$.
Setting $v_i(y)=\ei^{\frac{2}{p_i-1}}u_i(\psi_i(\ei y))$ for $y\in \Beilinha$ we know that $v_i$ satisfies 
\begin{align*}
\begin{cases}
L_{\hat{g}_i}v_i+n(n-2)v_i^{p_i}=0,&\text{in}\:\Beilinha,
\\
B_{\hat{g}_i}v_i-(n-2)Tv_i^{\frac{p_i+1}{2}}=0,&\text{on}\:\Deilinha,
\end{cases}
\end{align*}
where $\hat{g}_i$ is the metric with coefficients $(\hat{g}_i)_{kl}(y)=(g_i)_{kl}(\psi(\ei y))$.

In the case $\{x_i\}\subset M\backslash\partial M$, because the Fermi coordinates are centred at $\bar x_i$, by Proposition \ref{form:bolha} we have that $v_i\to U_T$ locally in $C^2$. In particular, 
$$
|\nabla v_i-\nabla U_T|(\xi_i)\to 0.
$$
Observe that $\partial U_T/\partial y_n(\xi_i)=(2-n)U_T(\xi_i)^{\frac{n}{n-2}}(T_i-T)$. 
Since the $x_i$ are critical points,
$$
(n-2)U_T(\xi_i)^{\frac{n}{n-2}}(T_i-T)=-\frac{\partial U_T}{\partial y_n}(\xi_i)=\big(\frac{\partial v_i}{\partial y_n}-\frac{\partial U_T}{\partial y_n}\big)(\xi_i)\to 0,\:\:\text{as}\:i\to\infty,
$$
so that $\lim_{i\to\infty}T_i=T$ and in particular $T\geq 0$. Also, if $T<0$ we necessarily have the case $\{x_i\}\subset \partial M$.

It remains to consider the case $T\geq 0$ with $\{x_i\}\subset\partial M$ and $\partial u_i/\partial \eta(x_i)\leq 0$.  
In this case, $T_i=0$, $\bar x_i=x_i$, $\xi_i=(0,...,0,0)$ and, according to Proposition \ref{form:bolha}, $v_i(y)\to \lambda^{\frac{2-n}{2}} U_T(\lambda^{-1}y)$ locally in $C^2$, where $\lambda=1/(1+T^2)$. Observe that 
$$
\frac{\partial}{\partial y_n}\Big|_{y=0}\lambda^{\frac{2-n}{2}} U_T(\lambda^{-1}y)=(n-2)U_T(0)^{\frac{n}{n-2}}\lambda T.
$$
Since 
$$
\Big|\frac{\partial v_i}{\partial y_n}-\frac{\partial}{\partial y_n}\lambda^{\frac{2-n}{2}} U_T(\lambda^{-1}y)\Big|(0)\to 0,
$$
the only possibility is $T=0$  and in particular $\lim_{i\to\infty}T_i=T$. 
\ep

A basic result for isolated simple blow-up point is stated as follows (see \cite[Proposition 1.4]{han-li1}).
\begin{proposition}\label{estim:simples}
	Let $x_i\to x_0$ be an isolated simple blow-up point for $\{u_i\}$, with $g_i\to g_0$. Then there exist $C,\delta>0$ such that
	\\\\
	(a) $u_i(x_i)u_i(x)\leq C\, d_{g_i}(x,x_i)^{2-n}$\:\:\: for all $x\in B_{\delta}(x_i)\backslash\{x_i\}$;
	\\\\
	(b) $u_i(x_i)u_i(x)\geq C^{-1}G_i(x)$\:\:\: for all $x\in B_{\delta}({x}_i)\backslash B_{r_i}({x}_i)$, where $G_i$ is the Green's function so that
	
\begin{align}
	\begin{cases}
	L_{g_i}G_i=0, &\text{in}\; B_{\delta}({x}_i)\backslash\{{x}_i\},\notag
	\\
	G_i=0, &\text{on}\; S_{\delta}({x}_i)\cap (M\backslash \partial M),\notag
	\\
	B_{g_i}G_i=0, &\text{on}\;(B_{\delta}({x}_i)\backslash\{{x}_i\} )\cap \partial M, \: \text{if} \: (B_{\delta}({x}_i)\backslash\{{x}_i\} )\cap \partial M\neq \emptyset, \notag 
	\end{cases}
	\end{align}
	and $d_{g_i}(x,{x}_i)^{n-2}G_i(x)\to 1$, as $x\to x_i$. Here, $r_i$ is defined as in Remark \ref{rk:def:equiv}.
\end{proposition}
\begin{remark}\label{rk:estim:simples}
Suppose that $x_i\to x_0$ is an isolated simple blow-up point for $\{u_i\}$. 
Set $$v_i(y)=\ei^{\frac{2}{p_i-1}}(u_i\circ\psi_i)(\e y)\,,\:\:\:\:\:\text{for}\: y\in B_{\delta\ei^{-1}}(x_i)\,,$$
where we are using geodesic coordinates $\psi_i$ centred at $x_i$. 
Then, as a consequence of  Propositions \ref{form:bolha} and \ref{estim:simples}, we see that $v_i\leq C(1+|y|)^{2-n}$ in $B_{\delta\ei^{-1}}(x_i)$.
\end{remark}


\section{Blow-up estimates}\label{sec:blowup:estim}
In this section we give a point-wise estimate for a blow-up sequence in a neighbourhood of an isolated simple blow-up point. 
As discussed in Section \ref{sec:intr}, in order to simplify our presentation we will only prove Theorem \ref{compactness:thm} as its extension to Theorem \ref{compactness:thm:general} can be done by a standard procedure (see for example \cite{almaraz-queiroz-wang}). To that reason, from this section to the end of our paper we will restrict our analysis to the case of critical exponent $p=(n+2)/(n-2)$. We will also fix the dimension $n=3$,  which is the content of our main results, so that $p=5$, $(p+1)/2=3$ and $T=-c$.

Let  $x_i\to x_0$ be an isolated simple blow-up point for the sequence $\{u_i\}$ of solutions to the equations \eqref{eq:blow-up}. In this section we assume $h_{g_i}=0$. This is because the results of this section will be called in the proof of Theorem \ref{compactness:thm} after a conformal change of the backgroud metric is applied. We will also assume  $x_0\in \partial M$ as the proof of Theorem \ref{compactness:thm}  in the case $x_0\in M\backslash\partial M$ follows from classical works (see for example \cite{li-zhu2}). 

As in the proof of  Lemma \ref{lem:bddTi}, for large $i$ choose $\bar x_i$ to be the point on $\partial M$ closest to $x_i$ and let $\psi_i: B^+_{\delta'}\to M$ denote Fermi coordinates centred at $\bar{x}_i$. We set $\xi_i=\epsilon_i^{-1}\psi_i(x_i)\in\mathbb R^n_+$ and write $\xi_i=(0,...,0,T_i)$ for a bounded sequence $T_i\geq 0$. We know that $\lim_{i\to\infty}T_i=T$ when $T\geq 0$, and $T_i=0$ when $T<0$, according to that lemma.
Setting $v_i(y)=\ei^{\frac{1}{2}}u_i(\psi_i(\ei y))$ for $y\in \Beilinha$ we know that $v_i$ satisfies 
\begin{align}\label{eq:vi'}
\begin{cases}
L_{\hat{g}_i}v_i+3v_i^{5}=0,&\text{in}\:\Beilinha,
\\
B_{\hat{g}_i}v_i-Tv_i^{3}=0,&\text{on}\:\Deilinha,
\end{cases}
\end{align}
where $\hat{g}_i$ is the metric with coefficients $(\hat{g}_i)_{kl}(y)=(g_i)_{kl}(\psi(\ei y))$. 

According to Proposition \ref{form:bolha} and Lemma \ref{lem:bddTi}, $v_i$ converges locally in $C^2$ norm to $U_T$ in the case $T\geq 0$, or to $\lambda^{\frac{2-n}{2}}U_T(\lambda^{-1}y)$ in the case $T<0$. 
We will consider only the case $T\geq 0$ in the rest of this section. For the case $T<0$, all the $U_T$ must be replaced by $\lambda^{\frac{2-n}{2}}U_T(\lambda^{-1}y)$ and the proof follows the same lines.

Let $r\mapsto 0\leq\chi(r)\leq 1$ be a smooth cut-off function such that $\chi(r)\equiv 1$ for $0\leq r\leq \delta'$ and $\chi(r)\equiv 0$ for $r>2\delta'$. 
We set $\chi_{\e}(r)=\chi(\e r)$.
Thus,  $\chi_{\e}(r)\equiv 1$ for $0\leq r\leq \delta'\e^{-1}$ and $\chi_{\e}(r)\equiv 0$ for $r>2\delta'\e^{-1}$.

The following result, although stated for dimension three, holds in any dimension $n\geq 3$ with the obvious modifications:

\begin{proposition}\label{Linearized}
Assume that $T\geq 0$. For every $i$ there is a solution $\phi_i$ of 
\begin{align}
	\begin{cases}\label{linear:8}
		\Delta\phi_{i}(y)+15U_T^4\phi_{i}(y)=-2\chi_{\ei}(|y|)\ei y_3 \pi_{kl}(\bar x_i)\displaystyle\frac{\d^2 U_T(y)}{\d y_k\d y_l}\,,&\text{for}\:y\in\R^3_+\,,
		\\
		\displaystyle\frac{\d\phi}{\d y_3} (\bar y)-3T U_T^2\phi_{i}(\bar{y})=0\,,&\text{for}\:\bar{y}\in\d\R^3_+\,,
	\end{cases}
	\end{align} 
where $\Delta$ stands for the Euclidean Laplacian, satisfying 
\begin{equation}\label{estim:phi'}
	|\nabla^r\phi_i|(y)\leq C\ei |\pi_{kl}(\bar x_i)|(1+|y|)^{-r}\,,\:\:\:\:\text{for}\: y\in\R^3_+\,,\,r=0,1 \:\text{or}\:2,
\end{equation} 
and
\begin{equation}\label{hip:phi}
	\phi_i(\xi_i)=1-U_T(\xi_i)\,,\:\:\frac{\d\phi_i}{\d y_1}(\xi_i)=\frac{\d\phi_{i}}{\d y_{2}}(\xi_i)=0.
\end{equation}
\end{proposition}

\bp
This proof follows the same lines as Proposition 5.1 of \cite{almaraz3} with some minor modifications and will be sketched here.
It strongly relies on the conformal equivalence $F_T:\mathbb S_T^3\to\mathbb R^3_+$, defined in Section \ref{sec:pre}.

Define, for each $i$,
$$
f_i(F_T^{-1}(y))=-2\chi_{\ei}(|y|)\ei y_3 \pi_{kl}(\bar x_i)\displaystyle\frac{\d^2 U_T(y)}{\d y_k\d y_l}U_T^{-5}(y),\:\:y\in\R^3_+.
$$
According to Lemma \ref{eigenvalues},
$$
z_1\circ F_T^{-1}=-2U_T^{-1}\frac{\partial U_T}{\partial y_1},
\qquad
z_2\circ F_T^{-1}=-2U_T^{-1}\frac{\partial U_T}{\partial y_2},
$$
and 
$$
z_3\circ F_T^{-1}=\frac{2}{\sqrt{{1+T^2}}}U_T^{-1}\big(\frac{1}{2}U_T+y^1\frac{\partial U_T}{\partial y_1}+y^2\frac{\partial U_T}{\partial y_2}+y^3\frac{\partial U_T}{\partial y_3}\big),
$$
so, by symmetry arguments,
$$
\int_{\mathbb S^3_T}z_1f_i=\int_{\mathbb S^3_T}z_2f_i=\int_{\mathbb S^3_T}z_3f_i=0.
$$
 Here, in the last equality we used that ${\text{tr}}(\pi_{kl}(\bar x_i))=2h_{g_i}(\bar x_i)=0.$
As the coordinate functions $z_1,z_2$ and $z_3$ generate the space of solutions of \eqref{linear:3'}, in particular for each $i$ one can find a smooth solution $\bar\phi_{\epsilon_i}$ of  
\begin{equation*}
\begin{cases}
\Delta_{\mathbb S^3} \bar{\phi}_{\epsilon_i} + 3\bar \phi_{\epsilon_i}= f_i\,,&\text{in}\:\mathbb S_T^3\,,\\
\displaystyle\frac{\partial \bar{\phi}_{\epsilon_i}}{\partial \eta}-T\bar{\phi}_{\epsilon_i}=0\,,&\text{on}\:\partial\mathbb  S_T^3\,,
\end{cases}
\end{equation*}
which is $L^2(\mathbb S^3_T)$-orthogonal to $z_1, z_2$ and $z_3$. We then consider the Green's function $G(z,w)$ on $\mathbb S_T^3$ such that
\begin{equation*}
\begin{cases}
\Delta_{\mathbb S^3} G + 3 G= \alpha(z_1w_1+z_2w_2+z_3w_3)\,,&\text{in}\:\mathbb S_T^3\,,\\
\displaystyle\frac{\partial G}{\partial \eta}-TG=0\,,&\text{on}\:\partial\mathbb  S_T^3\,,
\end{cases}
\end{equation*}
where $\alpha=\|z_1\|_{L^2(\mathbb S_T^3)}^{-2}=\|z_2\|_{L^2(\mathbb S_T^3)}^{-2}=\|z_3\|_{L^2(\mathbb S_T^3)}^{-2}$.

Observe that $\bar\phi_i=(2^{\frac{n-2}{2}}U_T\, \bar\phi_{\epsilon_i})\circ F_T^{-1}$ satisfies the required equations \eqref{linear:8} and a Green's representation formula argument gives the required estimates \eqref{estim:phi'}. Finally we choose
$$
\phi_i=\bar\phi_i+c_{1,i}J_1+c_{2,i}J_2+c_{3,i}J_3.
$$
Clearly, the coefficients $c_{1,i}, c_{2,i}$ and $c_{3,i}$ can be chosen so that the equations \eqref{hip:phi} hold, and $\phi_i$ satisfies \eqref{linear:8} because $J_1$, $J_2$ and $J_3$ are solutions to \eqref{linear:homog}. 
As $|c_{a,i}|\leq C\epsilon_i|\pi_{kl}(\bar x_i)|$, for $a=1,2,3$, it is easy to see that each $\phi_i$ satisfies  \eqref{estim:phi'}.

\ep

Now we are ready to obtain some refined estimates at isolated simple blow-up points. 
\begin{lemma}\label{estim:blowup1}
Assume that $T\geq 0$. There exist $\delta, C>0$ such that, for $|y|\leq\delta\ei^{-1}$,
\begin{equation}\label{eq:ei}
|v_i-U_T-\phi_i|(y)\leq C\ei\,.
\end{equation}
\end{lemma}

\begin{proof}
We consider $\delta<\delta'$ to be chosen later and set 
$$\Lambda_i=\max_{|y|\leq \delta\ei^{-1}} |v_i-U_T-\phi_i|(y)=|v_i-U_T-\phi_i|(y_i)\,,$$ 
for some $|y_i|\leq \delta\ei^{-1}$. From Remark \ref{rk:estim:simples}
we know that $v_i(y)\leq CU_T(y)$ for $|y|\leq \delta\ei^{-1}$. Hence, if there exists $c>0$ such that $|y_i|\geq c\ei^{-1}$, then
$$|v_i-U_T-\phi_i|(y_i)\leq C\,|y_i|^{-1}\leq C\,\ei.$$
This already implies the inequality \eqref{eq:ei} for $|y|\leq \delta\ei^{-1}$. Hence, we can suppose that $|y_i|\leq \delta\ei^{-1}/2$. 

Suppose, by contradiction, the result is false. 
Then, choosing a subsequence if necessary, we can suppose that
\begin{equation}
\label{hipLambda}
\lim_{i\to\infty}\Lambda_i^{-1}\ei=0\,.
\end{equation}
Define
$$w_i(y)=\Lambda_i^{-1}(v_i-U_T-\phi_i)(y)\,,\:\:\:\:\text{for}\:\: |y|\leq \delta\ei^{-1}\,.$$
By the equations \eqref{eq:U} and \eqref{eq:vi'}, $w_i$ satisfies
\begin{equation}\label{wi}
\begin{cases}
L_{\hat{g}_i}w_i+b_i w_i=Q_i\,,&\text{in}\:\Bei\,,\\
B_{\hat{g}_i}w_i+\bar b_i w_i=\overline{Q}_i\,,&\text{on}\:\Dei\,,
\end{cases}
\end{equation}
where 
$$b_i=3\frac{v_i^{5}-(U_T+\phi_i)^{5}}{v_i-(U_T+\phi_i)},
\qquad\bar b_i=-T\frac{v_i^{3}-(U_T+\phi_i)^{3}}{v_i-(U_T+\phi_i)},$$
$$Q_i=-\Lambda_i^{-1}\big\{3(U_T+\phi_i)^5-3U_T^5-15U_T^4\phi_i+(L_{\hat g_i}-\Delta)(U_T+\phi_i)-2\chi_{\ei}(|y|)\ei \pi_{kl}(\bar x_i)y_3(\d_k\d_l U_T)(y)\big\},$$
\begin{align*}
\overline{Q}_i&=\Lambda_i^{-1}\left\{T(U_T+\phi_i)^{3}
-TU_T^{3}-3TU_T^2\phi_i-\big(B_{\hat g_i}-\frac{\partial}{\partial y_3}\big)(U_T+\phi_i)\right\}
\\
&=\Lambda_i^{-1}\left\{T(U_T+\phi_i)^{3}
-TU_T^{3}-3TU_T^2\phi_i \right\}.
\end{align*}
Here, the last equality holds because we are using Fermi coordinates and $h_{\hat g_i}=0$ as a consequence of $h_{g_i}=0$.

Observe that, for any function $v$,
\begin{align}
(L_{\hat{g}_i}-\Delta)v(y)&=(\hat{g}^{kl}_i-\delta^{kl})(y)\d_k\d_l v(y)
+(\d_k\hat{g}^{kl}_i)(y)\d_l v(y)\notag
\\
&\hspace{2cm}-\frac{1}{8}R_{\hat{g}_i}(y)v(y)
+\frac{\d_k \sqrt{\det \hat{g}_i}}{\sqrt{\det \hat{g}_i}}\hat{g}_i^{kl}(y)\d_l v(y)\notag
\\
&=2\ei y_3\pi_{kl}(\bar x_i)\d_k\d_l v(y)+O(\ei^2(1+|y|)^{-1})\,.\notag
\end{align}
Hence,
\begin{align}\label{Qi}
Q_i(y)=O\left(\Lambda_i^{-1}\ei^2(1+|y|)^{-3}\right)+O\left(\Lambda_i^{-1}\ei^2(1+|y|)^{-1}\right)\,.
\end{align}
and
\begin{equation}\label{barQi}
\bar{Q}_i(\bar{y})= O\left(\Lambda_i^{-1}\ei^2(1+|\bar{y}|)^{-1}\right)\,.
\end{equation}
Moreover,
\begin{equation}
\label{lim:bi}
b_i\to 5U_T^4\,,\:\bar b_i\to -3U_T^2\,,\:\text{as}\:i\to\infty,\:\:\text{in}\: C^2_{loc}(\Rn)\,,
\end{equation}
and
\begin{equation}
\label{estim:bi}
b_i(y)\leq C(1+|y|)^{-4}\,,\:\:\bar b_i(y)\leq C(1+|y|)^{-2}\,,\:\:\:\:\text{for}\:\: |y|\leq\delta\ei^{-1}\,. 
\end{equation}

Since $|w_i|\leq |w_i(y_i)|=1$, we can use standard elliptic estimates to conclude that $w_i\to w$, in $C_{loc}^2(\mathbb{R}_+^n)$, for some function $w$, choosing a subsequence if necessary. From the identities \eqref{hipLambda}, \eqref{Qi}, \eqref{barQi} and  \eqref{lim:bi}, we see that $w$ satisfies
\begin{equation}\label{w}
\begin{cases}
\Delta w+15U_T^4w=0\,,&\text{in}\:\mathbb{R}_+^3\,,\\
\displaystyle\frac{\d w}{\d y_3}-3T U_T^{2}w=0\,,&\text{on}\:\partial\mathbb{R}_+^3\,.
\end{cases}
\end{equation}

\bigskip
\noindent
{\it{Claim.}} $\:\:w(y)=O((1+|y|)^{-1})$, for $y\in \Rn$.

\vspace{0.2cm}
Choosing $\delta>0$ sufficiently small, we can consider the Green's function $G_i$ for the conformal Laplacian $L_{\hat{g}_i}$ in $\Bei$ subject to the boundary conditions $B_{\hat{g}_i} G_i=0$ on $\Dei$ and $G_i=0$ on $\Sei$. Let $\eta_i$ be the inward unit normal vector to $\Sei$. Then the Green's formula gives
\begin{align}\label{wG}
w_i(y)&=\int_{\Bei}G_i(\xi,y)\left(b_i(\xi)w_i(\xi)-Q_i(\xi)\right) \,dv_{\hat{g}_i}(\xi)
+\int_{\Sei}\frac{\partial G_i}{\partial\eta_i}(\xi,y)w_i(\xi)\,d\sigma_{\hat{g}_i}(\xi)\notag
\\
&\hspace{1cm}
+\int_{\Dei}G_i(\xi,y)\left(\bar b_i(\xi)w_i(\xi)-\overline{Q}_i(\xi)\right)\,d\sigma_{\hat{g}_i}(\xi)\,.
\end{align}
Using the estimates \eqref{Qi}, \eqref{barQi} and \eqref{estim:bi}  in the equation \eqref{wG}, we obtain
\begin{align}
|w_i(y)|
\leq &\:C\int_{\Bei}|\xi-y|^{-1}(1+|\xi|)^{-4}d\xi
+C\Lambda_i^{-1}\ei^2\int_{\Bei}|\xi-y|^{-1}(1+|\xi|)^{-1}d\xi\notag
\\
&+C\int_{\Dei}|\bar{\xi}-y|^{-1}(1+|\bar{\xi}|)^{-2}d\bar{\xi}
+C\Lambda_i^{-1}\ei^2
\int_{\Dei}|\bar{\xi}-y|^{-1}(1+|\bar{\xi}|)^{-1}d\bar{\xi}\notag
\\
&+C\Lambda_i^{-1}\ei\int_{\Sei}|\xi-y|^{-2}d\sigma(\xi)\,,\notag
\end{align}
for $|y|\leq \delta\ei^{-1}/2$. Here, we have used the fact that $|G_i(x,y)|\leq C\,|x-y|^{-1}$ for $|y|\leq \delta\ei^{-1}/2$ and, since $v_i(y)\leq CU_T(y)$, $|w_i(y)|\leq C\Lambda_i^{-1}\ei$ for $|y|=\delta\ei^{-1}$. Hence, 
$$
|w(y)|\leq C\Lambda_i^{-1}\ei^2(\delta\ei^{-1})+C(1+|y|)^{-1}+C\Lambda_i^{-1}\ei^2\log(\delta\ei^{-1})
+C\Lambda_i^{-1}\ei.
$$
This gives
\begin{equation}\label{estim:wi}
|w_i(y)|\leq C\,\left((1+|y|)^{-1}+\Lambda_i^{-1}\ei\right)
\end{equation}
for $|y|\leq \delta\ei^{-1}/2$.
The Claim now follows from the hypothesis (\ref{hipLambda}).

Now, we can use the claim above and Lemma \ref{linear:homog} to see that 
$$w(y)=c_1\frac{\d U_T}{\d y_1}(y)+c_2\frac{\d U_T}{d y_2}(y)
+c_3\left(\frac{1}{2}U_T(y)+y^b\frac{\d U_T}{\d y_b}(y)\right)\,,$$
for some constants $c_1,c_2,c_3$.
It follows from the identities (\ref{hip:phi}) that 
$$w_i(\xi_i)=\frac{\partial w_i}{\partial y_1}(\xi_i)=\frac{\partial w_i}{\partial y_2}(\xi_i)=0.$$ Thus we conclude that $c_1=c_2=c_3=0$. Hence, $w\equiv 0$. Since $|w_i(y_i)|=1$, we have $|y_i|\to\infty$. This, together with the hypothesis (\ref{hipLambda}), contradicts the estimate (\ref{estim:wi}), since $|y_i|\leq \delta\ei^{-1}/2$, and concludes the proof of Lemma \ref{estim:blowup1}.
\ep


\begin{proposition}\label{estim:blowup:compl}
Assume that $T\geq 0$. There exist $C,\delta>0$ such that 
$$
|\nabla^k(v_i-U_T-\phi_i)(y)|\leq C\ei(1+|y|)^{-k}
$$ 
for all $|y|\leq\delta\ei^{-1}$ and $k=0,1,2$.
\end{proposition}
\bp
The estimate with $k=0$ is Lemma \ref{estim:blowup1}.
The estimates with $k=1,2$ follow from elliptic theory.
\ep

\begin{remark}
	As we have stated before, Proposition~\ref{estim:blowup:compl} holds for $T<0$ when we replace $U_T$ by $\lambda^{\frac{2-n}{2}}U_T(\lambda^{-1}y)$. The proof follows the same lines. 
\end{remark}

\section{The Pohozaev sign restriction and the mass term}\label{sec:sign:restr}

In this section we will first relate the mass type geometric invariant \eqref{def:mass} with the Pohozaev term $P'$ defined in Section \ref{sec:pre}. Then we will prove a sign restriction for an integral term in Proposition \ref{Pohozaev} and some consequences for the blow-up set. 

\begin{lemma}\label{propo:P'}
	Set $r=|z|$. If $\phi(z)=u(z)-r^{-1}$, then
	\begin{align*}
		P'(u,\rho)&=\frac{1}{2}\int_{S^+_{\rho}}\left(\frac{\partial}{\partial r}r^{-1}\phi-r^{-1}\frac{\partial \phi}{\partial r}\right)d\sigma
		\\
		&+\frac{1}{2}\int_{S^+_{\rho}}\left(r\,\Big(\frac{\d\phi}{\d r}\Big)^2-r|d\phi|^2+\frac{\d\phi}{\d r}\Big(\phi+r\frac{\d\phi}{\d r}\Big)\right) d\sigma\,.
	\end{align*}
\end{lemma}
\bp
Direct calculations give
\begin{align*}
	u\d_r u&-r|du|^2+2r(\d_r u)^2
	=(\d_r u)\Big(u+r\d_r u\Big)+r\Big((\d_r u)^2-|du|^2\Big)
	\\
	&=\d_r r^{-1}\phi-r^{-1}\d_r\phi+\phi\d_r\phi+r(\d_r\phi)^2+r\big((\d_r\phi)^2-|d\phi|^2\big),
\end{align*}
from which the result follows.
\ep

The asymptotically flat manifolds we work with in this paper come from the stereographic projection of compact manifolds with boundary. Inspired by Schoen's approach \cite{schoen1} to the classical Yamabe problem, this  projection is defined by means of a Green's function with singularity at a boundary point. Since we do not have the same control of the Green's function expression we do in the case of manifolds without boundary, the relation with \eqref{def:mass} is obtained by means of an integral defined in \cite{brendle-chen}. This is stated in the next proposition. 
\begin{proposition}\label{propo:I:mass}
	Let $(M,g)$ be a compact three-manifold with boundary and consider Fermi coordinates centred at $x_0\in \d M$.  Let $G$ be a smooth positive function on $M\backslash\{x_0\}$ written near $x_0$ as
	$$
	G(z)=|z|^{-1}+\phi(z)
	$$ 
	where $\phi$ is smooth on $M\backslash\{x_0\}$ satisfying $\phi(z)=O(|\log|z|\,|)$. If we define the metric $\hat g=G^{4}g$ and set 
	\begin{align*}
		I(x_0,\rho)
		&=8\int_{S^+_{\rho}}\left(|z|^{-1}\d_aG(z)-\d_a|z|^{-1}G(z)\right)\frac{z_a}{|z|}d\sigma
		\\
		&-12\int_{S^+_{\rho}}|z|^{-5}z_3z_i z_j\pi_{ij}(x_0)d\sigma\,,
	\end{align*}
	then $(M\backslash \{x_0\},\hat g)$ is asymptotically flat in the sense of Definition \ref{def:asym} with mass
	$$
	m(\hat g)=\lim_{\rho\to 0}I(x_0,\rho).
	$$
\end{proposition}
\bp
Consider inverted coordinates $y_a=|z|^{-2}z_a$.
The first statement follows from the fact that $\hat{g}\left(\frac{\d}{\d y_a},\frac{\d}{\d y_b}\right)=\delta_{ab}+O(|y|^{-1}|\log |y||)$.
In order to prove the last one, we can mimic the proof of \cite[Proposition 4.3]{brendle-chen} to obtain
\ba
\int_{S^+_{\rho^{-1}}}\frac{y_a}{|y|}\frac{\d}{\d y_b}
\hat{g}\left(\frac{\d}{\d y_a},\frac{\d}{\d y_b}\right)\ds_{\rho^{-1}}
&-\int_{S^+_{\rho^{-1}}}\frac{y_a}{|y|}\frac{\d}{\d y_a}
\hat{g}\left(\frac{\d}{\d y_b},\frac{\d}{\d y_b}\right)\ds_{\rho^{-1}}\notag
\\
&=\mathcal{I}(x_0,\rho)+O(\rho\, (\log\rho)^2)\,.\notag
\end{align}
Since $(z_1, z_2 ,z_3)$ are Fermi coordinates,
$$
\hat{g}\left(\frac{\d}{\d y_1}, \frac{\d}{\d y_3}\right)=\hat{g}\left(\frac{\d}{\d y_2}, \frac{\d}{\d y_3}\right)=0\,,
\:\:\:\text{if}\:y_3=0\,,
$$
the result then follows.
\ep

In the following proposition we relate the flux integral $I$ with the Pohozaev term $P'$.
\begin{proposition}\label{I:P'}
Under the hypotheses of Proposition \ref{propo:I:mass}, if $h_g=0$, we have
\begin{align*}
	P'(G,\rho)=-\frac{1}{16}I(x_0,\rho)+O(\rho\, |\log \rho|).
\end{align*}
\end{proposition}
\bp
Since $\text{tr}\,(\pi_{ij}(x_0))=2h_g(x_0)=0$, we have
\begin{align*}
-12\int_{S^+_{\rho}}&|z|^{-5}z_3z_i z_j \pi_{ij}(x_0)d\sigma
=0
\end{align*}
by symmetry arguments. A direct calculation shows
\begin{align*}
\int_{S^+_{\rho}}&\left(|z|^{-1}\d_aG(z)-\d_a|z|^{-1}G(z)\right)\frac{z_a}{|z|}d\sigma
\\
&=\int_{S^+_{\rho}}\left(|z|^{-1}\d_a\phi(z)-\d_a|z|^{-1}\phi(z)\right)\frac{z_a}{|z|}d\sigma\,,
\end{align*}
and so
\begin{align*}
I(x_0,\rho)=8\int_{\d^+B^+_{\rho}(0)}\left(|z|^{-1}\d_a\phi(z)-\d_a|z|^{-1}\phi(z)\right)\frac{z_a}{|z|}d\sigma\,.
\end{align*}
On the other hand, using Lemma \ref{propo:P'} we obtain
\begin{align*}
P'(G,\rho)=-\frac{1}{2}\int_{S^+_{\rho}}\left(|z|^{-1}\d_a\phi(z)-\d_a|z|^{-1}\phi(z)\right)\frac{z_a}{|z|}d\sigma+O(\rho\,|\log \rho|),
\end{align*}
and the result follows.
\ep

Now we are ready to obtain the Pohozaev sign condition to be used in the proof of our main Theorem~\ref{compactness:thm}.
\begin{theorem}\label{cond:sinal}
Let $x_i\to x_0\in \partial M$ be an isolated simple blow-up point for the sequence $\{u_i\}$ of solutions to the equations \eqref{eq:blow-up} with $h_{g_i}=0$.
Suppose that $u_i(x_i)u_i\to G$ away from $x_0$, for some function $G$. Then
\begin{equation}\label{eq:cond:sinal}
\liminf_{r\to 0}P'(G,r)\geq 0\,.
\end{equation}
\end{theorem}

\bp
Let $\bar x_i$ be the point on $\partial M$ closest to $x_i$.
We use Fermi coordinates $\psi_i=(z_1,...,z_n)$ centred at $\bar x_i$ and omit the symbol $\psi_i$ to simplify the notation. We define $v_i(y)=\ei^{\frac{1}{2}}u_i(\ei y)$ for $y\in \Bei$ and have 
\begin{align}\notag
\begin{cases}
L_{\hat{g}_i}v_i+3v_i^{5}=0,&\text{in}\:\Bei,
\\
B_{\hat{g}_i}v_i-Tv_i^{3}=0,&\text{on}\:\Dei,
\end{cases}
\end{align}
where $\hat{g}_i$ is the metric with coefficients $(\hat{g}_i)_{kl}(y)=(g_i)_{kl}(\ei y)$.
Observe that, from Remark \ref{rk:estim:simples}, we know that $v_i(y)\leq C(1+|y|)^{-1}$ in $B^+_{\delta\ei^{-1}}$.

According to Proposition \ref{form:bolha} and Lemma \ref{lem:bddTi}, $v_i$ converges locally in $C^2$ to $U_T$ in the case $T\geq 0$, or to $\lambda^{\frac{2-n}{2}}U_T(\lambda^{-1}y)$ in the case $T<0$. 
In order to simplify the notations, we will consider only the case $T\geq 0$ in the rest of this proof. For the case $T<0$ the only necessary modification would be replacing $U_T$ by $\lambda^{\frac{2-n}{2}}U_T(\lambda^{-1}y)$ in what follows.

If $\Delta$ is the Euclidean Laplacian, we set
$$
F_i(u,r)=-\int_{B_r^+}(z^b\partial_bu+\frac{1}{2}u)(L_{g_i}-\Delta)u\,dz
$$
and write the Pohozaev identity of Proposition \ref{Pohozaev} as
\begin{equation}\label{Pohoz}
P(u_i,r)=F_i(u_i,r)\,.
\end{equation}
Here, the integral on $D_r$ vanishes as we are using Fermi coordinates and assuming that $h_{g_i}=0$. 
Observe that 
\begin{align*}
F_i(u_i,r)=\int_{B_{r\epsilon_i^{-1}}^+}(y^b\partial_bv_i+\frac{1}{2}v_i)(L_{\hat{g}_i}-\Delta)v_idy.
\end{align*}
On the other hand, if $\check{U}_i(z)=\ei^{-\frac{1}{2}}U_T(\ei^{-1}z)$ and $\check{\phi}_i(z)=\ei^{-\frac{1}{2}}\phi_i(\ei^{-1}z)$, we have
\begin{align}
F_i(\check{U}_i+\check{\phi}_i,r)
&=-\int_{B_r^+}(z^b\partial_b(\check{U}_i+\check{\phi}_i)
+\frac{1}{2}(\check{U}_i+\check{\phi}_i)(L_{g_i}-\Delta)(\check{U}_i+\check{\phi}_i) dz\notag
\\
&={-}\int_{B_{r\epsilon_i^{-1}}^+}\left(y^b\partial_b(U_T+\phi_i)
+\frac{1}{2}(U_T+\phi_i)\right)\notag
(L_{\hat{g}_i}-\Delta)(U_T+\phi_i)dy\,.
\notag
\end{align}

It follows from Proposition \ref{estim:blowup:compl} that 
\begin{equation}\label{approx:F}
|F_i(u_i,r)-F_i(\check{U}_i+\check{\phi}_i,r)|\leq C \ei^2\int_{B_{r\epsilon_i^{-1}}^+}(1+|y|)^{-2}dy
\leq C\ei r\,.
\end{equation}
Using \eqref{exp:g} and recalling that $h_{g_i}=0$, due to symmetry arguments we have
$$
F_i(\check{U}_i+\check{\phi}_i,r)=O(\ei r)\,.
$$
Hence, $P(u_i,r)\geq -C\ei r$, which implies that
$$
P'(G,r)=\lim_{i\to \infty}\ei^{-1}P(u_i,r)\geq -Cr\,.
$$
The conclusion follows if we let $r\rightarrow 0$.
\ep

Once we have proved Theorem \ref{cond:sinal}, the next two propositions are similar to \cite[Lemma 8.2, Proposition 8.3]{khuri-marques-schoen} or \cite[Propositions 4.1 and 5.2]{li-zhu2}.
\begin{proposition}\label{isolado:impl:simples}
	 Let $x_i\to x_0$ be an isolated  blow-up point for the sequence $\{u_i\}$ of solutions to the equations \eqref{eq:blow-up} with $h_{g_i}=0$. Then $x_i\to x_0$ is an isolated simple blow-up point for $\{u_i\}$.
\end{proposition}
\begin{proposition} \label{dist:unif}
	Let $\b, R, u, C_0(\b,R)$ and $\{x_1,...,x_N\}\subset M$ be as in Proposition \ref{conj:isolados} with $p=(n+2)/(n-2)$ and $h_g=0$. If $\b$ is sufficiently small and $R$ is  sufficiently large, then there exists a constant $\bar C(\b,R)>0$ such that if  $\max_{M}u\geq C_0$ then 
	 $$d_{\bar g}(x_j,x_k)\geq \bar C\:\:\:\:\:\text{for all}\:1\leq j\neq k\leq N.$$
\end{proposition}
\begin{corollary}\label{Corol:8.4}
Suppose the sequence $\{u_i\}$ of solutions to the equations \eqref{eq:blow-up} with $h_{g_i}=0$ satisfies $\max_{M} u_i\to\infty$. Then the set of blow-up points for $\{u_i\}$ is finite and consists only of isolated simple blow-up points.
\end{corollary}


\section{The proof of Theorem \ref{compactness:thm}}\label{sec:pf:thm}
Without loss of generality we will assume that $h_g=0$ along $\partial M$ as one can simply change \eqref{main:eq:crit} by a conformal factor.

In view of standard elliptic estimates and Harnack inequalities, we only need to prove that $u$ is bounded from above (see \cite[Lemma A.1]{han-li1} for the boundary Harnack inequality). Assume by contradiction there exists a sequence $u_i$ of positive solutions of (\ref{main:eq:crit}) such that 
    \begin{equation*}
    \max_{M} u_i\to\infty,\:\:\:\:\text{as}\:i\to\infty.
    \end{equation*}
    It follows from Corollary \ref{Corol:8.4} that we can assume $u_i$ has $N$ isolated simple blow-up points 
    \begin{align*}
    x_i^{(1)}\to x^{(1)},\:...\:,\: x_i^{(N)}\to x^{(N)}.
    \end{align*}
     Without loss of generality, suppose 
     \begin{align*}
     u_i(x_i^{(1)})=\min\big\{u_i(x_i^{(1)}), ..., u_i(x_i^{(N)})\big\}\:\:\:\:\text{for all}\:i.
     \end{align*}
   
Now for each $k=1,...,N$, consider the Green's function $G_{(k)}$ for the conformal Laplacian $L_{g}$ with boundary condition $B_{g}G_{(k)}=0$ and singularity at $x^{(k)}\in M$. Observe that these Green's functions exist because $Q_c(M)>0$ by hypothesis.
In Fermi coordinates centred at the respective boundary singularities, or in normal coordinates at the interior ones, those functions satisfy
    \begin{align*}
    \big|G_{(k)}(z)-|z|^{-1}\big|\leq C\big|\log|z|\,\big|,
    \end{align*}
according to \cite[Proposition B.2]{almaraz-sun}.

It follows from the upper bound (a) of Proposition \ref{estim:simples} that there exists some function $G$ such that $u_i(x_i^{(1)})u_i\to G$ in $C^2_{\text{loc}}(M\backslash  \{x^{(1)},...,x^{(N)}\})$. Moreover, the lower control (b) of that proposition and elliptic theory yields the existence of $a_k>0$, $k=1,...,N$, and $b\in C^2(M)$ such that 
    \begin{equation*}
    G=\sum_{k=1}^{N}a_kG_{(k)}+b,
    \end{equation*}
 and 
    \begin{align*}
    \begin{cases}
    L_{g}b=0,&\text{in}\:M,
    \\
    B_{g}b=0,&\text{on}\:\d M.
    \end{cases}
    \end{align*}
The hypothesis $Q_c(M)>0$ ensures that $b\equiv 0$. 

Assume that $x^{(1)}$ lies on the boundary. Otherwise, the proof follows from the same arguments as in \cite{li-zhu2}.
If $\hat g=G_{(1)}^{4}g$, by Proposition \ref{propo:I:mass}, $(M\backslash\{x^{(1)}\}, \hat g)$ is an asymptotically flat manifold (in the sense of Definition \ref{def:asym}) with mass 
    $$
    m(\hat g)=\lim_{\rho\to 0} I(x^{(1)}, \rho).
    $$
Moreover, we have
$R_{\hat g}=-8G_{(1)}^{5}L_gG_{(1)}=0$ and $h_{\hat g}=-2G_{(1)}^{3}B_gG_{(1)}=0$.
Then the Positive Mass Theorem \ref{pmt} and the assumption that $M$ is not conformally equivalent to the hemisphere gives $m(\hat g)>0$. So, by  Proposition \ref{I:P'},
    $$
    \lim_{\rho\to 0}P'(G_{(1)}, \rho)<0.
    $$
     This contradicts the local sign restriction of Theorem \ref{cond:sinal} and ends the proof of Theorem \ref{compactness:thm}.


\bigskip\noindent
\textsc{S\'ergio Almaraz\\
Instituto de Matem\'atica e Estat\' istica, \\
Universidade Federal Fluminense\\
Rua Prof. Marcos Waldemar de Freitas S/N,
Niter\'oi, RJ,  24.210-201, Brazil}\\
e-mail: {\bf{sergioalmaraz@id.uff.br}}

\bigskip\noindent
\textsc{Shaodong Wang\\
	School of Mathematics and Statistics,\\
	Nanjing University of Science and Technology\\
	Nanjing, 210094, People’s Republic of China} \\
e-mail: {\bf{shaodong.wang@mail.mcgill.ca}}


\begin{thebibliography}{}

\bibitem{ahmedou}
Mohameden~Ould Ahmedou.
\newblock A {R}iemann mapping type theorem in higher dimensions. {I}. {T}he
conformally flat case with umbilic boundary.
\newblock In {\em Nonlinear equations: methods, models and applications
	({B}ergamo, 2001)}, volume~54 of {\em Progr. Nonlinear Differential Equations
	Appl.}, pages 1--18. Birkh\"{a}user, Basel, 2003.

\bibitem{almaraz-barbosa-lima}
S\'{e}rgio Almaraz, Ezequiel Barbosa, and Levi~Lopes de~Lima.
\newblock A positive mass theorem for asymptotically flat manifolds with a
non-compact boundary.
\newblock {\em Comm. Anal. Geom.}, 24(4):673--715, 2016.

\bibitem{almaraz-queiroz-wang}
S\'{e}rgio Almaraz, Olivaine~S. de~Queiroz, and Shaodong Wang.
\newblock A compactness theorem for scalar-flat metrics on 3-manifolds with
boundary.
\newblock {\em J. Funct. Anal.}, 277(7):2092--2116, 2019.

\bibitem{almaraz-sun}
S\'{e}rgio Almaraz and Liming Sun.
\newblock Convergence of the {Y}amabe flow on manifolds with minimal boundary.
\newblock {\em Ann. Sc. Norm. Super. Pisa Cl. Sci. (5)}, 20(3):1197--1272,
2020.

\bibitem{almaraz1}
S\'{e}rgio de~Moura Almaraz.
\newblock An existence theorem of conformal scalar-flat metrics on manifolds
with boundary.
\newblock {\em Pacific J. Math.}, 248(1):1--22, 2010.

\bibitem{almaraz5}
S\'{e}rgio de~Moura~Almaraz.
\newblock Blow-up phenomena for scalar-flat metrics on manifolds with boundary.
\newblock {\em J. Differential Equations}, 251(7):1813--1840, 2011.

\bibitem{almaraz3}
S\'{e}rgio de~Moura~Almaraz.
\newblock A compactness theorem for scalar-flat metrics on manifolds with
boundary.
\newblock {\em Calc. Var. Partial Differential Equations}, 41(3-4):341--386,
2011.

\bibitem{araujo}
Henrique Ara\'{u}jo.
\newblock Existence and compactness of minimizers of the {Y}amabe problem on
manifolds with boundary.
\newblock {\em Comm. Anal. Geom.}, 12(3):487--510, 2004.

\bibitem{berti-malchiodi}
Massimiliano Berti and Andrea Malchiodi.
\newblock Non-compactness and multiplicity results for the {Y}amabe problem on
{$S^n$}.
\newblock {\em J. Funct. Anal.}, 180(1):210--241, 2001.

\bibitem{brendle2}
Simon Brendle.
\newblock Blow-up phenomena for the {Y}amabe equation.
\newblock {\em J. Amer. Math. Soc.}, 21(4):951--979, 2008.

\bibitem{brendle-chen}
Simon Brendle and Szu-Yu~Sophie Chen.
\newblock An existence theorem for the {Y}amabe problem on manifolds with
boundary.
\newblock {\em J. Eur. Math. Soc. (JEMS)}, 16(5):991--1016, 2014.

\bibitem{brendle-marques}
Simon Brendle and Fernando~C. Marques.
\newblock Blow-up phenomena for the {Y}amabe equation. {II}.
\newblock {\em J. Differential Geom.}, 81(2):225--250, 2009.

\bibitem{chen}
Szu-Yu~Sophie Chen.
\newblock Conformal deformation to scalar flat metrics with constant mean curvature on the boundary in higher dimensions.
\newblock {\em ArXiv:0912.1302}.

\bibitem{chen-ruan-sun}
Xuezhang Chen, Yuping Ruan, and Liming Sun.
\newblock The {H}an-{L}i conjecture in constant scalar curvature and constant
boundary mean curvature problem on compact manifolds.
\newblock {\em Adv. Math.}, 358:106854, 56, 2019.

\bibitem{chen-sun}
Xuezhang Chen and Liming Sun.
\newblock Existence of conformal metrics with constant scalar curvature and
constant boundary mean curvature on compact manifolds.
\newblock {\em Commun. Contemp. Math.}, 21(3):1850021, 51, 2019.

\bibitem{chen-wu}
Xuezhang Chen and Nan Wu.
\newblock Blow-up phenomena for the constant scalar curvature and constant
boundary mean curvature equation.
\newblock {\em J. Differential Equations}, 269(11):9432--9470, 2020.


\bibitem{cherrier}
Pascal Cherrier.
\newblock Probl\`emes de {N}eumann non lin\'{e}aires sur les vari\'{e}t\'{e}s
riemanniennes.
\newblock {\em J. Funct. Anal.}, 57(2):154--206, 1984.

\bibitem{chipot-fila-shafrir}
M.~Chipot, I.~Shafrir, and M.~Fila.
\newblock On the solutions to some elliptic equations with nonlinear {N}eumann
boundary conditions.
\newblock {\em Adv. Differential Equations}, 1(1):91--110, 1996.

\bibitem{disconzi-khuri}
Marcelo~M. Disconzi and Marcus~A. Khuri.
\newblock Compactness and non-compactness for the {Y}amabe problem on manifolds
with boundary.
\newblock {\em J. Reine Angew. Math.}, 724:145--201, 2017.


\bibitem{druet2}
Olivier Druet.
\newblock Compactness for {Y}amabe metrics in low dimensions.
\newblock {\em Int. Math. Res. Not.}, (23):1143--1191, 2004.

\bibitem{escobar1}
Jos\'{e}~F. Escobar.
\newblock Uniqueness theorems on conformal deformation of metrics, {S}obolev
inequalities, and an eigenvalue estimate.
\newblock {\em Comm. Pure Appl. Math.}, 43(7):857--883, 1990.

\bibitem{escobar3}
Jos\'{e}~F. Escobar.
\newblock Conformal deformation of a {R}iemannian metric to a scalar flat
metric with constant mean curvature on the boundary.
\newblock {\em Ann. of Math. (2)}, 136(1):1--50, 1992.

\bibitem{escobar2}
Jos\'{e}~F. Escobar.
\newblock The {Y}amabe problem on manifolds with boundary.
\newblock {\em J. Differential Geom.}, 35(1):21--84, 1992.

\bibitem{escobar4}
Jos\'{e}~F. Escobar.
\newblock Conformal deformation of a {R}iemannian metric to a constant scalar
curvature metric with constant mean curvature on the boundary.
\newblock {\em Indiana Univ. Math. J.}, 45(4):917--943, 1996.

\bibitem{ahmedou-felli1}
Veronica Felli and Mohameden Ould~Ahmedou.
\newblock Compactness results in conformal deformations of {R}iemannian metrics
on manifolds with boundaries.
\newblock {\em Math. Z.}, 244(1):175--210, 2003.

\bibitem{ahmedou-felli2}
Veronica Felli and Mohameden Ould~Ahmedou.
\newblock A geometric equation with critical nonlinearity on the boundary.
\newblock {\em Pacific J. Math.}, 218(1):75--99, 2005.

\bibitem{ghimenti-micheletti1}
Marco Ghimenti and Anna~Maria Micheletti.
\newblock Compactness for conformal scalar-flat metrics on umbilic boundary
manifolds.
\newblock {\em Nonlinear Anal.}, 200:111992, 30, 2020.

\bibitem{ghimenti-micheletti2}
Marco~G. Ghimenti and Anna~Maria Micheletti.
\newblock A compactness result for scalar-flat metrics on low dimensional
manifolds with umbilic boundary.
\newblock {\em Calc. Var. Partial Differential Equations}, 60(3):Paper No. 119,
24, 2021.

\bibitem{han-li1}
Zheng-Chao Han and Yanyan Li.
\newblock The {Y}amabe problem on manifolds with boundary: existence and
compactness results.
\newblock {\em Duke Math. J.}, 99(3):489--542, 1999.
		
\bibitem{han-li2}
Zheng-Chao Han and YanYan Li.
\newblock The existence of conformal metrics with constant scalar curvature and
constant boundary mean curvature.
\newblock {\em Comm. Anal. Geom.}, 8(4):809--869, 2000.

\bibitem{khuri-marques-schoen}
M.~A. Khuri, F.~C. Marques, and R.~M. Schoen.
\newblock A compactness theorem for the {Y}amabe problem.
\newblock {\em J. Differential Geom.}, 81(1):143--196, 2009.

\bibitem{kim-musso-wei}
Seunghyeok Kim, Monica Musso, and Juncheng Wei.
\newblock Compactness of scalar-flat conformal metrics on low-dimensional
manifolds with constant mean curvature on boundary.
\newblock {\em Ann. Inst. H. Poincar\'{e} C Anal. Non Lin\'{e}aire},
38(6):1763--1793, 2021.


\bibitem{li-zhang}
Yan~Yan Li and Lei Zhang.
\newblock Compactness of solutions to the {Y}amabe problem. {II}.
\newblock {\em Calc. Var. Partial Differential Equations}, 24(2):185--237,
2005.

\bibitem{li-zhang2}
Yan~Yan Li and Lei Zhang.
\newblock Compactness of solutions to the {Y}amabe problem. {III}.
\newblock {\em J. Funct. Anal.}, 245(2):438--474, 2007.

\bibitem{li-zhu}
Yanyan Li and Meijun Zhu.
\newblock Uniqueness theorems through the method of moving spheres.
\newblock {\em Duke Math. J.}, 80(2):383--417, 1995.

\bibitem{li-zhu2}
Yanyan Li and Meijun Zhu.
\newblock Yamabe type equations on three-dimensional {R}iemannian manifolds.
\newblock {\em Commun. Contemp. Math.}, 1(1):1--50, 1999.

\bibitem{coda1}
Fernando~C. Marques.
\newblock Existence results for the {Y}amabe problem on manifolds with
boundary.
\newblock {\em Indiana Univ. Math. J.}, 54(6):1599--1620, 2005.

\bibitem{coda2}
Fernando~C. Marques.
\newblock Conformal deformations to scalar-flat metrics with constant mean
curvature on the boundary.
\newblock {\em Comm. Anal. Geom.}, 15(2):381--405, 2007.

\bibitem{marques}
Fernando~Coda Marques.
\newblock A priori estimates for the {Y}amabe problem in the non-locally
conformally flat case.
\newblock {\em J. Differential Geom.}, 71(2):315--346, 2005.
	
	
\bibitem{schoen1}
Richard Schoen.
\newblock Conformal deformation of a {R}iemannian metric to constant scalar
curvature.
\newblock {\em J. Differential Geom.}, 20(2):479--495, 1984.

\bibitem{schoen-yau}
Richard Schoen and Shing~Tung Yau.
\newblock On the proof of the positive mass conjecture in general relativity.
\newblock {\em Comm. Math. Phys.}, 65(1):45--76, 1979.



\bibitem{terracini}
Susanna Terracini.
\newblock Symmetry properties of positive solutions to some elliptic equations
with nonlinear boundary conditions.
\newblock {\em Differential Integral Equations}, 8(8):1911--1922, 1995.


\bibitem{witten}
Edward Witten.
\newblock A new proof of the positive energy theorem.
\newblock {\em Comm. Math. Phys.}, 80(3):381--402, 1981.


\end{thebibliography}
\end{document}